\documentclass[12pt,reqno]{amsart}

\usepackage[utf8]{inputenc}
\usepackage[T1]{fontenc}
\usepackage[svgnames,hyperref]{xcolor}
\usepackage{lmodern,amssymb,bbm}
\usepackage[UKenglish,USenglish,english,american]{babel}
\usepackage[left=3cm,right=3cm,top=3cm,bottom=3cm]{geometry}
\usepackage[cmtip]{xy}
\xyoption{pdf}
\xyoption{color}
\xyoption{all}

\usepackage[shortlabels]{enumitem}
\setlist[enumerate]{label=\rm{(\arabic*)}}
\setlist[enumerate,2]{label=\rm({\it\roman*})}
\setlist[itemize]{label=\raisebox{0.25ex}{\tiny$\bullet$}}

\theoremstyle{plain}    
 \newtheorem{thm}{Theorem}[section]
 \numberwithin{equation}{section} 
 \numberwithin{figure}{section} 
 \theoremstyle{plain}
 \theoremstyle{plain}    
 \newtheorem{cor}[thm]{Corollary} 
 \theoremstyle{plain}    
 \newtheorem{prop}[thm]{Proposition} 
 \theoremstyle{plain}    
 \newtheorem{lem}[thm]{Lemma} 
 \theoremstyle{remark}
 
 \theoremstyle{definition}

\newtheorem{notation}[thm]{Notation}

\newtheorem{thmA}{Theorem}

\theoremstyle{definition}
\newtheorem{defi}[thm]{Definition}

\newcommand{\C}{{\mathbb{C}}}

\newcommand{\R}{{\mathbb{R}}}

\newcommand{\e}{\varepsilon}

\newcommand{\f}{\varphi}

\newcommand{\dt}{\partial_t}

\DeclareMathOperator{\Sing}{Sing}

\DeclareMathOperator{\Amp}{Amp}

\newcommand{\mycolor}{Navy}
\usepackage[pdfauthor={Dat T. T\^O}, colorlinks, linktocpage, citecolor = \mycolor, linkcolor = \mycolor, urlcolor = \mycolor]{hyperref}
\usepackage[all]{hypcap} 

\title[Convergence of the  weak K\"ahler-Ricci Flow]{Convergence of the weak K\"ahler-Ricci Flow on  manifolds of general type}
\author[T. D. T\^O]{Tat Dat T\^O}
\date{\today}
\address{Ecole Nationale de l'Aviation Civile, Unversit\'e  de Toulouse\\
7, Avenue Edouard Belin\\
FR-31055 Toulouse Cedex 04, France}

\address{Institut Math\'ematiques de Toulouse \\ Universit{\'e} de Toulouse, CNRS, UPS
\\ 31062 Toulouse Cedex 09\\ France (Associated Researcher).}

\email{tat-dat.to@enac.fr, tat-dat.to@math.univ-toulouse.fr}

\begin{document}
\begin{abstract}  
We study the  K\"ahler-Ricci flow on compact K\"ahler manifolds whose canonical bundle is big. We show that  the normalized K\"ahler-Ricci flow has long time existence in the viscosity sense, is continuous in a  Zariski open set, and converges to the unique singular K\"ahler-Einstein metric in the canonical class.  The key ingredient is  a viscosity theory for degenerate complex Monge-Amp\`ere flows in  big classes that we develop, extending and refining the approach of Eyssidieux-Guedj-Zeriahi. 
\end{abstract}
\maketitle

\section*{Introduction}
%
Let $(X,\omega)$ be a compact K\"ahler manifold  of general type, i.e  the canonical bundle $K_X$ is big. 
We  study the normalized K\"ahler-Ricci flow on $X$:
\begin{equation}\label{NKRF_0}
\frac{\partial \omega_t}{\partial t}=-Ric(\omega_t)-\omega_t, \quad \quad
\omega_{|_{t=0}}=\omega_0.
\end{equation}
Let $T$ be the maximal existence time of the smooth flow. It is  known that $T=\infty$ if and only if the canonical bundle $K_X$ is nef, and in this case the normalized K\"ahler-Ricci flow converges to a singular K\"ahler-Einstein metric on $K_X$ (cf. \cite{Tsu88,TZ06}).   When $K_X$ is not nef, the flow has a finite time singularity ($T<\infty$). The limit class of the flow is
$$\{\alpha_T\}=\lim_{t\rightarrow T} \{\omega(t)\}=e^{-T}\{\omega_0\}+(1-e^{-T})c_1(K_X).$$
The class $\alpha_{T}$ is big and nef. For $t>T$, $\alpha_t:=\{\omega(t)\}$ is still  big but no longer nef, thus we can not continue the flow in the classical sense (we refer to \cite{SW, Tos} for  more details about   the  K\"ahler-Ricci flow).

  It was  asked by  Feldman-Ilmanen-Knopf \cite[Question 8, Section 10]{FIK} whether one can define and construct weak solutions of K\"ahler-Ricci flow after  the maximal existence time.
 In   \cite{ST12,ST}, Song and Tian have succeeded in repairing some  finite time singularities, defining weak solutions in the sense of pluripotential theory, by using strong algebraic results from the Minimal Model Program and by changing the underlying manifolds.  In \cite{BT12}, Boucksom and Tsuji have tried to run the weak  normalized K\"ahler-Ricci on projective varieties  beyond the maximal time using the  discretization of the K\"ahler-Ricci flow  and algebraic  tools.  They have proposed the following:

\medskip
\noindent {\bf Conjecture.}\cite[Conjecture 1, page 208]{BT12} { \it 
Let  $X$ be a compact K\"ahler manifold with pseudoeffective canonical bundle and $\omega_0$ be a K\"ahler form on  $X$. Then there exists a family of closed semipositive current $\omega(t)$ on $X$  such that 
\begin{enumerate}
\item $\{\omega(t)\}=e^{-t}\{\omega_0\}+(1-e^{-
t})c_1(K_X)$ and $\omega(0)=\omega_0$,
\item for any $T>0$, there exists a nonempty Zariski open subset $U(T)$ such that $\omega(t)$ is  a K\"ahler form on $U(T),\forall t\in [0,T)$,
\item  on $U(T)\times [0,T)$, $\omega(t)$ satisfies the normalized K\"ahler-Ricci flow (\ref{NKRF_0}). 
 \end{enumerate} }
 
\medskip
\noindent
 In this note we give an  answer to  the question of  a Feldman-Ilmanen-Knopf and study the conjecture of Boucksom-Tsuji. We moreover show that  the weak normalized K\"ahler-Ricci flow converges to the unique singular K\"ahler-Einstein metric in $K_X$ constructed in \cite{BEGZ,EGZ09}.   Our method is based on a viscosity  approach; It   does not use any deep algebraic technology and allows us to keep working on the same underlying manifold.  Precisely, we have the following theorem:
 
 \begin{thmA}\label{thm_KRF}
  {\it Let $(X,\omega_0)$ be a compact  K\"ahler manifold with $K_X$ is big. Fix $\theta$  a smooth $(1,1)$-form in $c_1(K_X)$.  Then the K\"ahler-Ricci flow starting from $\hat \omega \in\{\omega_0\}$ 
$$\frac{\partial \omega_t}{\partial t}=-Ric(\omega_t)-\omega_t$$
admits a unique viscosity solution $\omega_t=e^{-t}\omega_0+(1-e^{t})\theta+dd^c\f_t$ for all time. 

Moreover the flow converges,  exponentially fast in $\Amp(K_X)$ (see Definition \ref{ample_locus}), to the unique singular K\"ahler-Einstein metric in the canonical class $K_X$. And
\begin{itemize}
\item for $0<t<T_{\max}$, the function $x\mapsto \f_t(x)$ identifies with the smooth solution in \cite{TZ06}, where $T_{\max}$ is the maximal existence time of the smooth flow,
\item for $t\geq T_{\max}$ the flow $(\f_t)$ is continuous in $[T_{
max},+\infty)\times\Amp(K_X)$.
\end{itemize}
}
 \end{thmA}
 
 \medskip
 A compact K\"ahler manifold with $K_X$ big turns out to be projective, by a classical result of Moishezon (cf. \cite{Moi}). We actually prove a more general convergence result (see Theorem \ref{convergence of KRF}), valid in the K\"ahler setting.

 \medskip
Since the K\"ahler-Ricci flow can rewritten as a parabolic complex Monge-Amp\`ere equation, a key ingredient of our approach is to construct weak solutions  for degenerate complex  Monge-Amp\`ere flows

\begin{equation}\label{CMAF_big}
\quad(\theta_t+dd^c\f_t)^n=e^{\dot{\f_t}+F(t,x,\f_t)}\mu\quad \text{on }X_T:=[0,T)\times X,
\end{equation}
where 
\begin{itemize}
\item  $(\theta_t)_{t\in [0,T]}$ be a continuous  family of smooth closed $(1,1)$-forms such that $\alpha_t= \{\theta_t\}$ is {\it big},
\item $F(t,x,r)$ is a continuous in $ [0,T)\times X$ and non decreasing in $r$,
\item $\mu(x)=f(x)dV\geq 0$ is a    continuous volume form on $X$.
\end{itemize}

\medskip
 Degenerate complex elliptic Monge-Amp\`ere equations on compact K\"ahler manifold have recently been studied intensively using tools from pluripotential theory following the pioneering work of Bedford and Taylor in the local case ( cf. \cite{BT1,BT2,Kol98,GZ05,GZ07,BEGZ}).   
A pluripotential theory for the parabolic side only developed recently \cite{GLZ1,GLZ2}.

 A complementary viscosity approach for complex Monge-Amp\`ere equations has been developed  in \cite{EGZ11, EGZ17,EGZ15,HL09, Wan12}. The similar theory for the parabolic case  has been developed in \cite{EGZ15b} on complex domains (see also \cite{DLT} for its extension) and in \cite{EGZ16,EGZ18} on compact K\"ahler  manifolds.  It is very  interesting to compare the viscosity and plutipotential notions (we refer the reader to \cite{GLZ3} for more details). 
 
  For complex Monge-Amp\`ere flows, both theories have been developed when the involved class  $(\alpha_t)$ is big and semipositive.  For further applications,  we need to extend these theories in the general case where $(\alpha_t)$  is  not  necessarily  semipositive.  

\medskip
In the first part of the note, we  therefore establish a viscosity theory for degenerate complex Monge-Amp\`ere flows where the  involved classes  $(\alpha_t)$ are {\it big}, not necessarily semipositive,  extending the results  in \cite{EGZ16,EGZ18}. We refer the reader to Section \ref{sect:adapted_defi} for the adapted definition of viscosity subsolution (resp. supersolution). The first result is a   general viscosity  comparison principle as follows.

\begin{thmA} \label{thm_comparison}
{\it Assume that  $\theta_t\geq \theta$ for a smooth $(1,1)$-form $\theta$ in some fixed big class $\alpha$. Let $\f$ (resp. $\psi$) be a viscosity subsolution (resp.  supersolution) to (\ref{CMAF_big}). 
Then  for any  $ (t,x)\in  [0,T)\times \Amp(\alpha)
$
$$(\f-\psi)(t,x)\leq \max\{ \sup_{\{0\}\times X}(\f-\psi)^* ,0\}  .$$}
\end{thmA}

This comparison principle  not only  generalizes   previous results to the case of big cohomology classes but also refines the K\"ahler case,  since  our assumption is weaker than these  previous works (the authors needed some conditions on either $\partial_t \varphi$ or $\theta_t$, cf.  \cite[Theorem 2.1]{EGZ16},  \cite[Theorem 4.2]{EGZ18}). 
In the present work, we exploit the  concavity of  $\log \det$ to overcome the difficulties in \cite{EGZ16,EGZ18}. Moreover, Theorem \ref{thm_comparison} can be extended to adapt  the  involved classes of the K\"ahler-Ricci flow (we refer the reader to Corollary \ref{cor:comparison}). 
  
\medskip
As a first application of the comparison principle,  we study the Cauchy problem  on a compact K\"ahler manifold $(X,\omega)$
$$(CP_1)\quad \begin{cases}
 (\theta+dd^c\f_t)^n=e^{\dot{\f_t}+\f_t}\mu \\
\f(0,x)=\f_0, 
\end{cases}  
$$
where 
\begin{itemize}
\item $\theta$ is a smooth $(1,1)$-form in a fixed big class,
\item  $\f_0$ is an $\theta$-psh function with minimal singularities which is continuous in $\Amp(\{\theta\})$\item $\mu=fdV>0$ is a continuous volume form on $X$.   
\end{itemize}

\medskip
There exists a unique solution to the static (elliptic) equation 
\begin{equation}\label{eq:elliptic}
(\theta+dd^c\varphi)^n=e^\varphi \mu
\end{equation}
by \cite{BEGZ}.

\medskip
We first prove the existence of  viscosity subsolutions and supersolutions to $(CP_1)$ and construct barriers at each point $\{0\}\times \Amp(\{\theta\})$. We then use the Perron method to show the existence of a unique viscosity solution:

\begin{thmA}\label{thm_Cauchy_1}
{\it The exists a unique viscosity solution to  $(CP_1)$ in $[0,\infty)\times \Amp(\{\theta\})$. Moreover, the flow  asymptotically recovers the solution of the elliptic Monge-Amp\`ere equation \eqref{eq:elliptic}. }
\end{thmA}
In the second part, we apply our techniques and study the normalized K\"ahler-Ricci flow on compact K\"ahler  manifolds  whose the canonical bundle  is {\it big}, i.e {\sl manifolds of general type}. We first use the viscosity theory above to construct the weak flow for all time.  We then prove the convergence  result of the flow, finishing the proof of Theorem \ref{thm_KRF}. In particular,  Theorem \ref{thm_Cauchy_1} is used to construct a viscosity supersolution which gives a uniform upper bound to the potential of the flow in $\Amp(K_X)$.

\medskip
The paper is organized as follows. In Section \ref{pre} we recall some notations of the viscosity theory on compact K\"ahler manifolds. In Section  \ref{sec:comparison_principle} we define the viscosity sub/super solutions for Complex Monge-Amp\`ere flows on big classes,   and  prove Theorem \ref{thm_comparison}. As a first application of Theorem  \ref{thm_comparison}, we prove  Therem \ref{thm_Cauchy_1} in Section \ref{Cauchy_pro_sect}.  Finally we prove the existence and convergence of the normalized K\"ahler-Ricci flow on compact K\"ahler manifolds of general type in Section \ref{KFR_sect}. 

\medskip
\noindent\textbf{Acknowledgement.} The author is grateful to  Vincent Guedj for support, suggestions and encouragement.  We also would like to thank   S\'ebastien Boucksom, Hoang Son Do, Henri Guenancia and Ahmed Zeriahi for fruitful discussions.  We would like to  thank Hoang-Chinh Lu for very useful discussions, suggestions and his encouragement to write down a missing argument in the proof of Theorem \ref{existence of KRF}. The author would like to thank the referees for very useful comments and suggestions. This work is supported  partially by the project  ANR GRACK.
\section{Preliminary}\label{pre}
 \subsection{Monge-Amp\`ere operator in big cohomology classes}
 
\subsubsection{Big cohomology classes}
Let $(X,\omega)$ be a compact K\"ahler manifold and let $\alpha\in H^{1,1}(X,\R)$ be a real $(1,1)$-cohomology class. 
\begin{defi}
The class $\alpha$ is  {\it pseudo-effective} if it can be represented by a closed positive $(1,1)$-current $T$.  Moreover $\alpha$ is called {\it big}  if the  $(1,1)$-current $T$  can be chosen to be {\it strictly positive}, i.e $T$ dominates some smooth positive form on $X$. 
\end{defi}
\begin{defi}
The class $\alpha$ is  {\it nef} if it lies in the closure of the  K\"ahler cone, i.e the convex cone containing  all K\"ahler classes. 
\end{defi}
 Let $T,T'$ be two positive closed current in $\alpha$ with the local potentials $\f,\f'$ respectively. We say that $T$ is less singular than $T'$ if  $\f'\leq \f+O(1)$. In addition, $T$  is said to have {\it minimal singularities} if it is less singular  than any other positive current in $\alpha$. 

One important example for such a current is the following. We first pick $\theta\in \alpha$ a smooth representative, then  the upper envelope
$$V_\theta:=\sup\{\varphi; \varphi\in PSH(X,\theta)\text{ and } \sup_X\varphi\leq 0\}$$ 
yields  a current $\theta+dd^c V_\theta$ with minimal singularities (see \cite{BD12} for the regularity of this current).

\begin{defi} A positive closed current $T$ has {\it analytic singularities} if it can be locally written $T=dd^cu$, with 
$$u=\frac{c}{2}\log \sum |f_j|^2 +v,$$
where $c>0$, $v$ is smooth and the $f_j's$ are holomorphic functions.
\end{defi}
\begin{defi}\label{ample_locus}
If $\alpha$ is a big class, we denote Amp$(\alpha)$ the {\it ample locus} of $\alpha$, i.e. the set of all $x\in X$ for which there exists a K\"ahler current in $\alpha$ with analytic singularities which is smooth in a neighborhood of $x$.
\end{defi}
By definition the ample locus is a Zariski open subset and it is non-empty by Demailly's regularization result \cite{Dem92}. It follows from \cite{Bou04} that there exists a strictly positive current $T=\theta+dd^c\psi \in \alpha$ with analytic singularities such that 
$$\Amp(\alpha)=X\setminus \Sing T, \text{ and } T\geq C\omega,$$
for some $C>0$.
\begin{lem}\cite{Bou04}\label{rho}
There exists a $\theta$-psh function $\rho$ with such that 
\begin{enumerate}
\item $\theta+dd^c\rho\geq \epsilon \omega_X$, for some $\epsilon>0$,
\item $\rho$ is smooth in $\Amp(\alpha)$ and $\Amp(\alpha)=\{ \rho=-\infty\}$,
\item $\rho\leq V_\theta$, 
\item $\rho(z) -V_\theta(z)\rightarrow -\infty$ as $z\rightarrow \partial \Amp(\alpha)$. 
\end{enumerate}
\end{lem} 

\subsubsection{Non-pluripolar product}
Let $X$ be an $n$-dimensional complex manifold. Let $u_1,\ldots,u_p$ be psh functions on $X$. Denote 
$$\mathcal{O}_k:=\cap_{j=1}^{p}\{u_j>-k\},$$
\begin{defi}
If $u_1,\ldots,u_p$ are psh functions on $X$, we  say that the non-pluripolar product $\langle \bigwedge_{j=1}^p dd^cu_j\rangle$ is well-defined on $X$ if for each compact subset $K$ of $X$ we have
\begin{equation}
\sup_k\int_{K\cap O_k}\omega^{n-p}\cap \bigwedge dd^c \max\{u_j,-k\}< \infty.
\end{equation}
\end{defi}

Now let $(X,\omega)$ be a compact K\"ahler manifold and $\alpha $ be a big cohomology class on $X$. Given  a $\theta$-psh function $\f$, we can define its non-pluripolar Monge-Amp\`ere by ${\rm MA}(\f):=\left\langle (\theta+dd^c\f)^n\right\rangle $.  Then we have 

$$\int_X \left\langle (\theta+dd^c\f)^n\right\rangle \leq {\rm vol}(\alpha).$$
We  say that the function $\f$ such that the equality holds has {\sl full Monge-Amp\`ere mass}. In particular, all $\theta$-psh functions with minimal singularities have full Monge-Amp\`ere mass. 

\medskip
\noindent
{\bf Notation.} From now we denote the non-pluripolar Monge-Amp\`ere  product $(\theta+dd^c\f)^n$  instead of $\left\langle (\theta+dd^c\f)^n\right\rangle$. 
\subsection{Monge-Amp\`ere equation in big cohomology classes}
\subsubsection{Pluripotential approach}
Fix $\alpha$ a big class on $X$ and $\theta$  a smooth $(1,1)$ form representing $\alpha$.  We consider the following Monge-Amp\`ere equation 
\begin{equation}\label{eq:elliptic_1}
(\theta+dd^c\f)^n=fdV
\end{equation}
where $f\in L^{p}(X)$ for some $p>1$. Then we have the following theorem of existence of solution due to \cite[Theorem 4.1]{BEGZ}
\begin{thm}
There exists  a unique solution $\f\in PSH(X,\theta)$ to the equation  (\ref{eq:elliptic_1}) satisfying $\sup_X \f=0$. Moreover,  there exists a constant $M$ only depending on $\theta,dV, p$ such that 
$$\f\geq V_\theta-M \|f \|^{1/n}_{L^p}.$$
\end{thm}

In that same paper \cite{BEGZ}, the authors also established the existence of solutions to the equation 
\begin{equation}
(\theta+dd^c\f)^n=e^{\f}\mu
\end{equation}
with $\mu$ is a smooth volume form.  This implied the existence of a unique singular K\"ahler-Einstein metric on $K_X$. More precisely, we have
\begin{thm}\cite{BEGZ}\label{thm:BEGZ2}
Let $\mu=fdV$ be a volume form with $f\in L^p(X)$ for some $p>1$. Then there exists a unique $\theta$-psh function $\f$ such that 
\begin{equation}\label{eq:elliptic_2}
(\theta+dd^c \f)^n=e^{\f}\mu.
\end{equation}
Furthermore, $\f$ has minimal singularities.  
\end{thm}
We refer the reader to \cite{GZ17,PS} for more details about complex Monge-Amp\`ere equations. 
\subsubsection{Viscosity approach}
The pluripotential theory gives us the existence of $\theta$-psh solution with minimal singularities to the equation (\ref{eq:elliptic_2}). 
In \cite{EGZ15}, the authors developed a viscosity theory for the  complex Monge-Amp\`ere equation \eqref{eq:elliptic_2}  in which $\mu$ is continuous. They   proved the existence of a unique viscosity solution to (\ref{eq:elliptic_2}) which is {\it continuous} in $\Amp(\alpha)$. Furthermore, this is exactly the pluripotential solution.

\medskip
We recall here the basic results in \cite{EGZ15} on the viscosity approach to the equation: 
\begin{equation}\nonumber
(MA_\mu) \hskip1cm  (\theta+dd^c\varphi)^n=e^{\varphi}\mu,
\end{equation}
where $\theta$ is a smooth $(1,1)$ form representing $\alpha$. Denote $\Omega=\Amp(\alpha)$.
 \begin{defi}(Test functions)
Let $\f:X\rightarrow \R$ be any function and $x_0$ a given point such that $\f(x_0)$ is finite. An {\sl upper test function} (resp. a {\sl lower test function}) for $\f$ at $x_0$ is a $C^{2}$- function $q$ in a neighborhood of $x_0$ such that $\f(x_0)=q(t_0,x_0)$ and $\f\leq q$ (resp. $\f\geq q$) in a neighborhood of $x_0$. 
\end{defi}

\begin{defi}
A function $\f: X \rightarrow \R\cup\{-\infty\}$ is a {\sl viscosity subsolution} of $(MA_\mu)$ on $X$ if 
\begin{itemize}
\item $\f: \Omega\rightarrow \R$ is upper semi-continuous,
\item $\f \leq V_{\theta}+C$  on $X$, for some constant $C$ and $\f\not\equiv - \infty$,
\item for any point $x_0\in \Omega$ and any upper test function $q$ for $\f$ at $x_0$, we have 
$$(\theta(x_0) +dd^c q(x_0))^n\geq e^{q(x_0)}\mu(x_0).$$
\end{itemize}
\end{defi}

\begin{defi}
A function $\psi: X\rightarrow \R\cup\{-\infty,+\infty\}$ is a {\sl viscosity supersolution} of $(MA_\mu)$ on $X$ if 
\begin{itemize}
\item $\psi: \Omega\rightarrow \R$ is lower semi-continuous,
\item $\psi\geq V_{\theta} -C$ on $X$, for some $C>0$ and $\f\not\equiv +\infty$,
\item for any point $x_0\in\Omega$ and any lower test function $q$ for $\psi$ at $x_0$, we have 
$$(\theta(x_0) +dd^c q(x_0))_+^n\leq e^{q(x_0)}\mu(x_0).$$
\end{itemize}
\end{defi}
Here $\omega_{+}=\omega$ if $(1,1)$-form $\omega$ is semipositive and $\omega_{+}=0$  otherwise.  

\begin{defi}
A {\sl viscosity solution} of $(MA_\mu)$ is a function that is both a viscosity subsolution and viscosity supersolution. In particular, a viscosity solution has the same singularities as $V_\theta$.
\end{defi}

\begin{thm}\label{CP}
Let $\alpha$ be a big cohomology class. Let $\f$ (resp. $\psi$) be a viscosity subsolution (resp. supersolution) of $(MA\mu)$, then 
$$\f\leq \psi \text{ in } \Amp(\alpha).$$
\end{thm}
We sketch a proof following \cite[Theorem 13.11]{GZ17} in which the cohomology class $\alpha$ is semipositive and big.  This  is slightly different from the proof in \cite{EGZ15}. 

\begin{proof}
Since $\alpha=\{\theta\}$ is big,  there exists a $\theta$-psh function $\rho \leq \f$  satisfying
$$\theta+dd^c\rho\geq C\omega.$$
Moreover, it follows from Lemma \ref{rho} that one can find $\rho$ to be smooth in the ample locus $\Omega=\Amp(\alpha)$  such that $(\rho(x)-V_\theta) \rightarrow -\infty$ as $x\rightarrow \partial \Omega$. 

\medskip 
Now, fix $\lambda\in (0,1)$ and set 
$\tilde{\f}=(1-\lambda)\f +\lambda \rho.$
By the definitions   of sub/super-viscosity solutions $\f-\psi$ is bounded from above on $X$, hence $\tilde \f-\psi$ is  also bounded from above on $X$, so  we can extend it as an usc function $(\tilde{\f}-\psi)^*$ on $X$. Since $(\tilde{\f}-\psi)= (1-\lambda)(\f-V_\theta)-(\psi-V_\theta)+\lambda(\rho-V_\theta)$ is usc in $\Omega$ and tends to $-\infty$ as $x\rightarrow \partial \Omega$, the maximum of $(\tilde{\f}-\psi)^*$ is achieved at some point $x_0$ in $\Omega$,
$$\sup_{x\in X}(\tilde{\f}-\psi)^*  =\tilde{\f}(x_0)-\psi(x_0).$$
We now need prove that $\tilde{\f}(x_0)\leq \psi(x_0)$. 
The  proof of this claim is similar as in  \cite[Theorem 13.11]{GZ17}. Finally we have  $\tilde \psi\leq \psi$, and letting $\lambda\rightarrow 0$ we get the required inequality. 
\end{proof}
As a corollary of the comparison principle we have: 
\begin{thm}\cite{EGZ15}
Let $\alpha$ be a a big cohomology class and $\mu> 0$ is a continuous  density. Then there exists a unique pluripotential solution $\f$ of $(MA_\mu)$ on $X$, such that 
\begin{enumerate}
\item $\f$ is a $\theta$-psh function with minimal singularities,
\item $\f$ is a viscosity solution in $\Amp(\alpha)$ hence continuous here,
\item Its lower semicontinuous regularization $\f_*$ is a viscosity supersolution.
\end{enumerate} 
\end{thm}    

\section{Degenerate complex Monge-Amp\`ere flows in big classes}
\label{sec:comparison_principle}
In this section we define viscosity solutions to degenerate complex Monge-Amp\`ere flows in big cohomology classes following \cite{CIL92,EGZ11,EGZ15,EGZ15b,EGZ16,EGZ18}. Our goal is to establish a general comparison principle for viscosity subsolutions and supersolution to the Monge-Amp\`ere flows in big cohomology classes. This extends some results in \cite{EGZ16,EGZ18} where the authors studied the case when the involved cohomology classes are semipositive and big. In particular, we do not assume any condition on either the derivative of subsolution  or the involved form $\theta_t$.  

\subsection{Viscosity subsolutions and supersolutions}\label{sect:adapted_defi}
Let $X$ be a $n$-dimensional compact K\"ahler manifold. Fix $\alpha$ is a big cohomology class and $\theta$ is a smooth $(1,1)$-form representing $\alpha$. Denote by $\Omega=\Amp(\alpha)$ the ample locus of $\alpha$. We extend some definitions  from the viscosity theory for  complex Monge-Amp\`ere flows developed in \cite{EGZ16, EGZ18}. 

\medskip
We consider the following {\sl degenerate complex Monge-Amp\`ere flow}
$$(\theta_t+dd^c\phi_t)^n=e^{\dot{\phi}_t + F(t,x,\phi_t)}\mu\quad (CMAF),$$
where 
\begin{itemize}
\item $F(t,x,r)$ is a continuous in $ X_T=[0,T)\times X$ and non decreasing in $r$.
\item $\mu(t,x)\geq 0$ is a family of bounded continuous volume forms on $X$,
\item $\theta_t(x)$ is a family of  smooth  $(1,1)$-forms representing big cohomology classes $\alpha_t$ such that $\Amp(\alpha_t)\supset \Omega$.
\item $\phi: [0,T)\times X\rightarrow \R$ is the unknown function with $ \phi_t(\cdot):=\phi(t,\cdot)$.
\end{itemize}
  
 \begin{defi}(Test functions)
Let $\f:X_T\rightarrow \R$ be any function and $(t_0,x_0)$ a given point such that $\f(t_0,x_0)$ is finite. An {\sl upper test function} (resp. a {\sl lower test function}) for $\f$ at $(t_0,x_0)$ is a $C^{1,2}$- function $q$ in a neighborhood of $(t_0,x_0)$ such that $\f(t_0,z_0)=q(t_0,x_0)$ and $\f\leq q$ (resp. $\f\geq q$) in a neighborhood of $(t_0,x_0)$. 
\end{defi}

\begin{defi}
A function $\f: X_T\rightarrow \R\cup\{-\infty\}$ is a {\sl viscosity subsolution} of $(CMAF)$ on $X_T$ if 
\begin{itemize}
\item $\f: \Omega_T\rightarrow \R$ is upper semi-continuous,
\item $\f(t,x) \leq V_{\theta_t}(x)+C, \forall (t,x)\in X_T $  for some  $C$ possibly depending on $T$. 
\item for any point $(t_0,x_0)\in \Omega_T$ and any upper test function $q$ for $\f$ at $(t_0,x_0)$, we have 
$$(\theta_{t_0}(x_0) +dd^c q(t_0,x_0))^n\geq e^{\dot{q}(t_0,x_0)+F(t_0,x_0,q(t_0,x_0))}\mu(t_0,x_0).$$
\end{itemize}
\end{defi}

\begin{defi}\label{def:super}
A function $\psi: X_T\rightarrow \R\cup\{-\infty,+\infty\}$ is a {\sl viscosity supersolution} of $(CMAF)$ on $X_T$ if 
\begin{itemize}
\item $\psi: \Omega_T\rightarrow \R$ is lower semi-continuous,
\item $\psi(t,x)\geq V_{\theta_t}(x) -C,\forall (t,x)\in X_T$, for some $C>0$ possibly depending on $T$. 
\item for any point $(t_0,x_0)\in\Omega_T$ and any lower test function $q$ for $\psi$ at $(t_0,x_0)$, we have 
$$(\theta_{t_0}(x_0) +dd^c q(t_0,x_0))_+^n\leq e^{\dot{q}(t_0,x_0)+F(t_0,x_0,q(t_0,x_0))}\mu(t_0,x_0),$$
where $\omega_{+}=\omega$ if the $(1,1)$-form $\omega$ is semipositive and $\omega_{+}=0$  otherwise. 

\end{itemize}
\end{defi}

\begin{notation}
 From now, we  also write  $\f_t$  for  a function $\f$ depending on $t$, i.e $\f_t(x)=\f(t,x)$, and $\partial_t \f_t$ or $\dot \f_t$ for its derivative in the time variable. 
\end{notation}
\begin{defi}
A {\sl viscosity solution} of $(CMAF)$ on $X_T$ is a function that is both a viscosity subsolution and a viscosity supersolution. In particular,  a viscosity solution is  continuous in $\Omega_T$ and have the same singularities with $V_{\theta_t}$ on $X_T$. 
\end{defi}

\begin{lem}\label{pluri and vis  subsolutions}
Let $u$ is a $\theta_t$-psh with minimal singularities such that 
\begin{itemize}
\item $u$ is continuous in $\Omega_T$.
\item $u$ admits a continuous partial $\partial_t u$ with respect to $t$.
\item for any $t\in (0,T)$ the restriction of $u_t$ of $u$ to $X_t=\{t\}\times X$ satisfies
$$(\theta_t+dd^c u_t)^n\geq e^{\dot{u}_t+ F(t,x,u_t)}\mu$$
in the pluripotential sense on $\Omega_t$.

\end{itemize}
\medskip
   
Then $u$ is a subsolution of $(CMAF)$ in $\Omega_T$ . 
\end{lem}
\begin{proof}
Suppose $q$ is a test function of $u$ at $(t_0,x_0)\in [0,T)\times\Amp(\alpha)$. Then we have $q(t_0,x)$ is also a test function to $u_{t_0}(x)$ at $x_0$. Moreover, by the hypothesis, $u_{t_0}(x)$ satisfies
$$(\theta_{t_0}+dd^cu_{t_0})^n(x)\geq \nu(x)$$
in the pluripotential sense, where $\nu=e^{(\partial_t  u)(t_0,x)+F(t_0,x,u(t_0,z,x)}\mu(t_0,x)$ is a volume form with continuous density in $\Amp(\alpha)$. Therefore, by \cite[Theorem 1.9]{EGZ11}, we get 
\begin{align*}
(\theta+dd^c q)^n(t_0,x_0)&\geq \nu(x_0)= e^{\partial_t u (t_0,x_0)+u(t_0,x_0)+F(t_0,x_0,u(t_0,x_0))}\mu(t_0,x_0)\\
&\geq e^{\partial_t q(t_0,x_0))+F(t_0,x_0,q(t_0,x_0))}\mu(t_0,x_0).
\end{align*}
Hence $u$ is a viscosity subsolution of $(CMAF)$.
\end{proof}

\begin{lem}\label{pluri and vis supersolution}
Let $v: \Omega_T\rightarrow \R\cup \{-\infty,+\infty\}$ be a a lsc function satisfying
\begin{itemize}
\item The restriction $v_t$ of $v$ to $\Omega_t=\{t\}\times \Omega$ is $\theta_t$-psh function with minimal singularities.
\item $v$ admits a continuous partial derivative $\partial_t v$ with respect to $t$.
\item there exists a function $w$ on $X$ such that $w$ is continuous on $\Omega$ and $\partial_tv_t+F(t,x,v)\geq w $. Moreover, $w$ satisfies
$$(\theta+dd^cv_t)\leq e^w \mu$$
in the pluripotential sense on $\Omega_t$. 
\end{itemize}
Then $v$ is a viscosity supersolution of $(CMAF)$ in $\Omega_T$.
\end{lem}
\begin{proof}
Fix $(t_0,x_0)\in \Omega_T$. By the hypothesis, we have $$(\theta_{t_0}+dd^cv_{t_0})^n(x)\leq\nu(x),$$ where $\nu:=e^w\mu(t_0,x)$ is a volume form with continuous density in $\Omega$. It follows from \cite[Lemma 4.7]{EGZ11} that $v_{t_0}$ is a viscosity supersolution of the equation $(\theta+dd^c u)^n(x)=\nu(x)$ in $\Omega$. 

\medskip 
Now suppose $q$ is a lower test function of $v$ at $(t_0,z_0)\in \Omega_T$. Then $q(t_0,x_0)$ is also a lower test function for $v_{t_0}(x)$ at $x_0$. Therefore
\begin{align*}
(\theta_{t_0}+dd^c q)^n(t_0,x_0)&\leq e^{w(x_0)}\nu(x_0)\\
&\leq e^{(\partial_tv_t)(t_0,x_0)+F(t_0,x_0,v_{t_0})}\mu(t_0,x_0).
\end{align*}
Hence $v$ is a viscosity supersolution of $(CMAF)$ in $\Omega_T$.
\end{proof}

We show that subsolutions to parabolic Monge-Amp\`ere flows are plurisubharmonic in space variable.
\begin{prop}
Let $\f$ is a viscosity subsolution of $(CMAF)$. For each $t\in (0,T)$ we have $\f_t \in PSH(\Omega,\theta_t)$.
\end{prop}
\begin{proof}
Observe first that the problem is local. For any $x_0\in \Omega$, we can choose a small neighborhood  $U$ of $x_0$ such that $\theta_t=dd^c h_t$ for all $ (t,x)\in  (t-\e,t+\e)\times U$ for some $\e>0$ sufficiently small.  We then infers  that $u=h+\f$    is a viscosity subsolution of the local equation
\begin{equation}
(dd^cu_t)^n=e^{\dot{u}+\tilde{F}(t,z,u_t)}\mu
\end{equation}
on $(t-\e,t+\e)\times U$, where $\tilde{F}(t,z,u_t)=F(t,z, u_t-h_t) -\dot h_t$.  It follows from \cite[Corollary 3.7]{EGZ15} that $u_{t_0}$ is plurisubharmonic on $U$ for any $t_0\in (t-\e,t+\e)$.  Therefore we have $\f_t\in PSH(\Omega,\theta_t)$ as required.  
 \end{proof}

\subsection{A useful local comparison principle}
We recall here a useful lemma due to \cite[Corollary 3.9]{EGZ18} for the local equation
\begin{equation*}
 (dd^cu_t)^n=e^{\dot{u}_t+F(t,z,u_t)}\mu(t,z),  \quad\quad (MAF)_{F,\mu}
\end{equation*}
with the initial condition $u(0,z)=u_0$ a continuous psh function in $D\Subset\C^n$. 

\begin{lem}\label{local comparison}
Assume that $\mu(t,z)\geq 0$ be a continuous family of volume forms on some domain $D\Subset \mathbb{C}^n$. Let $u:[0,T)\times D\rightarrow \R$ be a viscosity subsolution to the local equation $(MAF)_{F,\mu}$ and let $v:[0,T)\times D\rightarrow \R$ be a supersolution to the local equation $(MAF)_{G,\mu}$. Assume that
\begin{itemize}
\item the function $u_t-v_t$ achieves a local maximum at some $(t_0,z_0)\in (0,T)\times D$.
\item there exits a constant $c_1>0$ such that $z\mapsto u(t,z)-2c_1|z^2|$ is a plurisubhamonic near $z_0$ and $t$ near $t_0$.
\end{itemize}
If either $\mu (t_0,z_0)>0$ or $\mu=\mu(z)$, then 
\begin{equation}
F(t_0,z_0,u(t_0,z_0))\leq G(t_0,z_0,v(t_0,z_0)).
\end{equation}
\end{lem}

\subsection{Comparison principle for complex Monge-Amp\`ere flows}
Let $(X,\omega)$ be a compact K\"ahler manifold. Fix $\alpha$ is a big cohomology class and $\theta$ is a smooth $(1,1)$-form representing $\alpha$. 
We consider the following  degenerate complex Monge-Amp\`ere flow
$$(\theta_t+dd^c\phi_t)^n=e^{\dot{\phi}_t + F(t,x,\phi_t)}\mu\quad (CMAF),$$
where 
\begin{itemize}
\item $F(t,x,r)$ is a continuous in $ X_T=[0,T)\times X$ and non decreasing in $r$.
\item $\mu=fdV\geq 0$ is a   continuous volume form on $X$,
\item $\theta_t(x)$ is a family of  smooth  $(1,1)$-forms representing big cohomology classes $\alpha_t$ such that $\Amp(\alpha_t)\supset \Omega:= \Amp(\alpha)$, 
\item $\phi: [0,T)\times X\rightarrow \R$ is the unknown function with $ \phi_t(\cdot):=\phi(t,\cdot)$.
\end{itemize}

\medskip
We now prove the following comparison principle  extending the one in \cite{EGZ16,EGZ18} where the class $\{\theta_t\}$ is semipositive and big.  In particular, we exploit the concavity of $\log\det$ avoiding the difficulties from the  time derivative  of the subsolution. 

\begin{thm}\label{comparison_strict_positive}
 Assume that  $\theta_t\geq \theta$ for a smooth $(1,1)$-form $\theta$ in some fixed big class $\alpha$. Let $\f$ (resp. $\psi$) be a viscosity subsolution (resp.  a supersolution) to $(CMAF)$. 
Then  for any  $ (t,x)\in  [0,T)\times \Amp(\alpha)
$
$$(\f-\psi)(t,z)\leq \max\{ \sup_{\{0\}\times X}(\f-\psi)^* ,0\}  $$
 \end{thm}

\begin{proof} 
Since $\alpha=\{\theta\}$ is big,  by Lemma \ref{rho},  there exists a $\theta$-psh function $\rho \leq \f$  satisfying
$$\theta+dd^c\rho\geq C\omega,$$
$\rho$ is smooth in the ample locus $\Amp(\alpha)$ and  $(\rho(x)-V_\theta) \rightarrow -\infty$ as $x\rightarrow \partial \Omega$. Since $\theta_t\geq \theta$, we have $V_\theta\leq V_{\theta_t}$, so  $(\rho(x)-V_{\theta_t}) \rightarrow -\infty$ as $x\rightarrow \partial \Omega$.

\medskip
Now fix $\delta>0$, $\lambda \in (0,1)$  and set
$$\f_\lambda(t,x):=(1-\lambda)\f(t,x)+\lambda \rho(x) -\frac{\delta}{T-t}-A\lambda t \leq \f,$$
where $A$ will be chosen hereafter.  Then $\f_\lambda$ is a strictly $\theta_t-$psh function since
$\theta+dd^c\f_t\geq \lambda (\theta+dd^c\rho)\geq C\omega$.
We now prove that $\f_\lambda$ satisfies 
\begin{equation}\label{subsol_phi_lambda}
(\theta_t+dd^c \f_\lambda)^n\geq e^{\dt \f_\lambda + (1-\lambda )F(t,x,\f_\lambda)+n \lambda \log C +A + \frac{\delta}{(T-t)^2} } \mu
\end{equation}
in the viscosity sense. Indeed, let $q$ be  an upper test for $\f_\lambda$ at $(t_0,z_0)$ then 
$$\tilde{q}= (1-\lambda)^{-1}\left(q-\lambda \rho(x) +\frac{\delta}{T-t}+A\lambda t \right)$$ 
is an upper test for $\f$ at $(t_0,z_0)$. If $\mu(z_0)=0$ then we have already $(\theta_t+dd^c \tilde{q})^n\geq 0 $ at $(t_0,z_0)$ hence   $(\theta_t+dd^c q)^n\geq 0 =\mu(z_0)$. Now assume that  $\mu(z_0)\neq 0$. Using the concavity of $\log\det$ and  $\theta_t+dd^c\rho \geq\theta+dd^c \rho\geq  C\omega_X$,  we have at  $(t_0,z_0)$
\begin{eqnarray*}
\log\frac{(\theta_t+ dd^c q)^n}{\mu} &=&\log\frac{ \left((1-\lambda)(\theta_t+dd^c \tilde q) +\lambda (\theta+dd^c\rho)\right)^n}{\mu}\\
&\geq & (1-\lambda )\log\frac{ (\theta_t+dd^c \tilde q)^n}{\mu} +\lambda \log\frac{ (\theta_t+dd^c\rho)^n}{\mu}\\
&\geq & (1-\lambda )\log\frac{ (\theta_t+dd^c \tilde q)^n}{\mu} +n\lambda\log C.
\end{eqnarray*}
It follows from the definition of viscosity subsolution  and the inequality above that   at  $(t_0,z_0)$
\begin{eqnarray*}
\dt q &=&  (1-\lambda)\dt  \tilde{q} +A\lambda+\frac{\delta}{(T-t)^2}\\
&\leq &(1-\lambda) \left(\log \frac{(\theta_t+dd^c \tilde q)^n}{\mu} -F(t_0,z_0, \tilde{q}(t_0,z_0))\right) -A\lambda-\frac{\delta}{(T-t)^2}\\
&\leq &  \log\frac{(\theta_t+ dd^c q)^n}{\mu}  -(1-\lambda )F(t_0,z_0, \tilde{q}(t_0,z_0) )-n\lambda\log C-A\lambda-\frac{\delta}{(T-t)^2}\\
&\leq&  \log\frac{(\theta_t+ dd^c q)^n}{\mu}  -(1-\lambda )F(t_0,z_0, q(t_0,z_0)) -n\lambda\log C-A\lambda-\frac{\delta}{(T-t)^2},
\end{eqnarray*} 
where the last inequality comes from the fact that $F$ is non-decreasing in the third variable and $\f_\lambda\leq \f$.   This implies \eqref{subsol_phi_lambda} as required. 

\medskip
By the definition of viscosity sub/super solutions, ${\f}_\lambda-{\psi}\leq \f-\psi$ is bounded from above, we can extend it as an usc function $({\f}_\lambda-\psi)^*$ on $X$. Moreover
\begin{eqnarray}
\f_\lambda-\psi =(1-\lambda )(\f-V_{\theta_t})-(\psi-V_{\theta_{t}})+\lambda(\rho-V_{\theta_t})  -\frac{\delta}{T-t}-A\lambda t,
\end{eqnarray}
hence $\f_\lambda-\psi   \rightarrow -\infty$ as  either $x\rightarrow \partial \Omega$ or $t\rightarrow T$. It follows that there exists $(t_0, x_0)\in [0,T)\times \Omega$ such that 
$$\sup_{X_T} (\f_\lambda-\psi )^*=\f_\lambda (t_0,x_0)-{\psi}(t_0,x_0). $$

 The idea is to localize near $x_0$ and use Lemma \ref{local comparison}. We choose the complex coordinates $z=(z_1,\ldots,z_n)$ near $x_0$ defining a biholomorphism identifying a closed neighborhood $\bar{B}$ of $x_0$ to the closed complex ball $\bar{B}_3:=B(0,3)\subset \C^n$ of radius 3, sending $x_0$ to the origin in $\C^n$.
Let $h_t$ be a smooth local potential for $\theta_t$ in $B_2$, i.e. $dd^ch_t=\theta_t$ in $B_3$.

\medskip
We have
$$\max_{[0,T)\times z^{-1}(\bar{B}_2)} (\f_\lambda-\psi)=\f_\lambda(t_0,x_0) - {\psi}(t_0,x_0).$$
If $t_0=0$, we are done. Otherwise,  assume  $t_0\in (0,T)$, we now prove that $\f_\lambda\leq \psi$ in $[0,T]\times \Omega $. 

\medskip
Now $ u(t,z)=\f_\lambda\circ z^{-1}+h_t\circ z ^{-1}$ is upper semi-continuous in $[0,T)\times B_2$ and strictly psh in $B_2$ since $dd^c u_t\geq \lambda C\omega_X$.  It follows  from \eqref{subsol_phi_lambda}  that 
$$(\theta_t+dd^c \f_\lambda)^n\geq e^{\dt \f_\lambda + (1-\lambda )F(t,x,\f_\lambda)+n \lambda \log C+A +\frac{\delta}{(T-t)^2} } \mu$$
 in $(t_0-r,t_0 +r)\times B_2$.  Therefore 
$$ (dd^c u_t)^n\geq e^{\partial_t u+\tilde{F}(t,z, u_t)}\tilde{\mu},\, \text{in } B_2,$$
where  $\tilde{\mu}=z_*\mu\geq 0$ is a  continuous volume form and 
$$\tilde{ F}(t,z,s)=(1-\lambda)F(t,z, s-h(t,z))+ n\lambda \log C+\frac{\delta}{(T-t)^2}+ A\lambda-\partial_t h_t(z)  .$$

\medskip Similarly, we have
$v=\psi\circ z^{-1}+h_t\circ z^{-1}$ is lower semi-continuous in $[0,T)\times B_2$ and satisfies 
$$(dd^c  v)^n_+\leq e^{\partial_t \bar v+\tilde{G}(t,z, v(t,z))}\tilde{\mu}, \, \text{in } B_2,$$
where $$\tilde{G}(t,z,s):=F(t,z,s-h_t(z))-\partial_t h_t(z)  .$$

By our assumption we have
$$\max_{[0,T)\times \bar{B}_2} u - v=  u(t_0,0) - v(t_0,0).$$
 It follows from Lemma \ref{local comparison} that 
$$\tilde{F} (t_0,0,  u (t_0,0))  \leq  \tilde G(t_0,0, v (t_0,0)) ,$$
hence
 $$F(t_0,x_0,\f_\lambda) +\lambda(A-F(t_0,x_0,\f_\lambda) +n\log C)+\frac{\delta}{(T-t)^2} \leq  F(t_0,x_0, \psi(t_0,x_0)) .$$
Choosing $A=-n\log C+ \sup_{X_T} F(t,x,\f)$, we have 
$$F(t_0, \f_\lambda(t_0,x_0)) + \frac{\delta}{(T-t)^2}\leq F(t_0,x_0,\psi(t_0,x_0)).$$
This implies that $\f_\lambda(t_0, x_0)\leq \psi(t_0,x_0)$, hence $ \f_\lambda\leq {\psi}$ in $[0,T)\times X $. Letting $\lambda \rightarrow 0$ and $\delta\rightarrow 0$ we get ${\f}\leq  {\psi}$ in $[0,T)\times\Omega$, we thus conclude that $\f\leq \psi$ in $[0,T)\times\Omega$ as required.
\end{proof}

\begin{cor}\label{cor:comparison}
With the same assumption above, but replacing the condition $\theta_t\geq \theta$ by $\theta_t\geq g(t)\theta$  for some  smooth positive function $g:[0,T]\rightarrow \R$ with $g'>0$. Then 
 if $ \f(0,x)\leq \psi(0,x), \forall x\in X$, 
 we have
$$\f(t,x)\leq \psi(t,x), \forall (t,x)\in  \Omega_T.$$
\end{cor}

\begin{proof}
Denote $\omega_t(z)=g(t)^{-1}\theta_t$.
 Since $\f$ is a subsolution to $(CP)$, we have
$$(\omega_t+dd^c\tilde{\f}_t)^n\geq e^{g(t) \partial_t \tilde{\f}+ \tilde F(t,z,\tilde{\f})} \mu,$$
in the sense of viscosity, where $\tilde{\f} (t,x):= g(t)^{-1}\f(t,x)$ and $\tilde F(t,z,s) =F(t,z,s)+ g'(t)s$.

\medskip
 We change the time variable: $$\chi(t,x)=\omega (h(t),x) \quad \text{ and }\quad \phi(t,x):=\tilde{\f } (h(t),x)= \frac{1}{g(h(t))}\f(h(t),x),$$ where $h$ will be chosen hereafter. Then $$\partial_t \phi (t,x)=(\partial_t \tilde{\f}) (h(t),x)h'(t),$$
 hence
\begin{eqnarray}
(\chi(t,x)+dd^c \phi(t,x))^n\geq e^{ \frac{g(h(t))}{h'(t)} \partial _t \phi  +\tilde{F}(t,z, \phi) }\mu.
\end{eqnarray}
We choose $h$ such that $h'(t)=g(h(t))$ and $h(0)=0$, hence $\phi(t,x)$ is a subsolution of 

\begin{equation}\label{new eq}
(\chi_t+dd^c \phi)^n=e^{\partial_t\phi+\tilde{F}(t,x,\phi) }\mu.
\end{equation}

Similarly, we have $\hat{\psi}(t,x)=\frac{1}{g(h(t))}\psi(h(t),x)$ is a supersolution of (\ref{new eq}).  Since $\chi\geq \theta$, Theorem \ref{comparison_strict_positive} thus implies the desired inequality. 
\end{proof}

\subsection{Viscosity solutions to Cauchy problem for complex Monge-Amp\`ere flows}
Let $X$ be a $n$-dimensional compact K\"ahler manifold and $\alpha$ be a fixed big class on $X$. We have a general Cauchy problem on $X_T= [0,T)\times X$
$$(CP)\hspace*{1cm} \begin{cases}
 (\theta_t+dd^c\f_t)^n=e^{\dot\f_t+ F(t,x,\f_t) }\mu(t,z), \\ 
 \f(0,x)=\f_0(x)\quad  x\in X 
\end{cases},
$$
where  
\begin{itemize}
\item $F(t,x,r)$ is a continuous in $[0,T)\times X$ and non decreasing in $r$.
\item $\mu(t,x)\geq 0$ is a family of bounded continuous volume forms on $X$,
\item $\theta_t(x)$ is a family of  smooth  $(1,1)$-forms representing big cohomology classes $\alpha_t$ such that $\Amp(\alpha_t)\supset \Omega:= \Amp(\alpha)$.
\item $\varphi_0$ is a given $\theta_0$-psh function. 
\end{itemize}
\begin{defi}
A {\it subsolution} to $(CP)$ is a vicosity subsolution $u$ to the flow
$$(\theta_t+dd^c\f_t)^n=e^{\dot  \f_t +F(t,x, \f_t) }\mu,$$ 
on $X_T$,  satisfying that $u(0,x)\leq \f_0(x)$ for all $x\in X $.

\medskip
A {\it supersolution} to $(CP)$ is a vicosity supersolution $v$ to the flow
$$(\theta_t+dd^c\f_t)^n=e^{\dot \f_t+ F(t,x, \f_t) }\mu,$$
on $X_T$,
satisfying  $v(0,x)\geq \f_0(x)$ for all $x\in X$.
\end{defi}
 
 \begin{defi}
 A function $u$ on $X$ is a vicosity solution to the Cauchy problem $(CP)$ if it is both a subsolution and a supersolution for $(CP)$. 
 \end{defi}

\section{Cauchy problem in a big cohomology class}
\label{Cauchy_pro_sect}
Let $(X,\omega)$ be a K\"ahler manifold.  Fix $\alpha$ is a big cohomology class on $X$ and $\theta$ is a smooth $(1,1)$-form representing $\alpha$. 
 In this section we consider the following Cauchy problem 
$$ (CP_1)\hspace*{1cm} \begin{cases}
(\theta+dd^c\f_t)^n= e^{\dot{\f_t}+\f_t}\mu \quad \text{in }\, [0,T)\times X,\\
\f(0,x)=\f_0(x),\quad (0,x)\in \{0\}\times X,
\end{cases}
$$ 
where  $\mu=\mu(x)>0$ is a positive continuous volume form  and $\f_0$ is a given $\theta$-psh function on $X$ with minimal singularities which is continuous in $\Omega:=\Amp(\alpha)$. 
We first have the following lemma which  is useful to construct subsolutions. 
\begin{lem}\label{lem:ODE}
Let $f: \R^+\rightarrow \R$ be the smooth solution of the ODE
\begin{equation}
f'(t)+f(t) =C\log (1-Be^{-t}) \quad \text{with} \quad f(0)=0, 
\end{equation}
for some $B,C>0$. 
There exists $A>0$ such that for all $t\geq 0$, $-A(t+1)e^{-t} \leq f(t)\leq 0$.
 \end{lem}

\subsection{Existence of viscosity sub/super-solutions}

\begin{lem}\label{subsolution 1}
$(CP_1)$ admits a viscosity subsolution.
\end{lem}
\begin{proof}
It flows from \cite{Bou04} that there exists a $\theta$-psh function $\rho$ with analytic singularities  satisfying
$$
\theta+dd^c\rho\geq \epsilon \omega  \quad \text{and }\sup_X\rho=0.
$$
 Set  
 $$u(t,x):=e^{-t}\f_0 +(1-e^{-t})\rho -At +f(t)$$ such that  the function $f$ satisfies the ODE: $f'(t)+f(t)= n\log (1-e^{-t})$ and  $f(0)=0$. 
Then  $u$ is continuous on $[0,T]\times  \Omega$ and    $u(0,x)=
 \f_0(x)$ and
\begin{eqnarray*}
(\theta+dd^cu)^n&=&\left( e^{-t}(\theta+dd^c\f_0)+(1-e^{-t})(\theta+dd^c\rho)\right)^n\\
&\geq& (1-e^{-t})^n(\theta+dd^c\rho)^n\\
&\geq &e^{n\log (1-e^{-t})} \epsilon^n \omega^n  \\
&=& e^{\partial_t u+u_t -\rho +A  +\log(\epsilon C) }\mu,\\
&\geq & e^{\partial_t u+u_t +A  +\log(\epsilon C) }\mu,
\end{eqnarray*} 
where $C^n\mu\leq \omega^n$. By choosing $A=-\log(\epsilon C)$, we infer that 
$ (\theta+dd^cu)^n\geq e^{\dot{u}_t +u_t}\mu$
in the pluripotential sense. 
Lemma \ref{pluri and vis  subsolutions} implies that $u$ is a viscosity subsolution of $(CP_1)$. 
\end{proof}

\begin{lem}\label{supersolution 1}
There exists a viscosity supersolution of $(CP_1)$.
\end{lem}
\begin{proof}
We can assume that $\theta\leq \omega$ for some K\"ahler form $\omega$.  Let $\phi$ is the unique continuous $\omega$-psh function satisfying
$$(\omega+dd^c \phi )^n=c\mu,\quad \sup_X \phi =0.$$
where $$c=\frac{\int_X \omega^n}{\int_X \mu}>0.$$  
Set $v=\phi+C$ with $C>0$ such that  $\f_0\leq v$ and $  \log c\leq v $. Then we have 
\begin{eqnarray}
(\theta +dd^c v)^n\leq (\omega+dd^c\phi)^n= c\mu \leq e^{\partial_t v+v}\mu.
\end{eqnarray}
It follows from Lemma \ref{pluri and vis supersolution} that $v$
 is a supersolution to $(CP_1)$. 
 \end{proof}

\subsection{Barrier construction}

\begin{defi}
Fix $(0,x_0)\in \{0\}\times \Omega$ and $\e>0$.
\begin{itemize}
\item An upper semi-continuous function $u:\Omega_T\rightarrow \R$ is an {\sl $\e$-subbarrier} to the $(CP_1)$ at the boundary point $(0,x_0)$, if $u$ is subsolution to $(CP_1)$  in $[0,T)\times X$, and $u_*(0,x_0)\geq \f_0(x_0)-\e$.\\
When $\e=0 $, $u$ is called a subbarrier.
\medskip
\item A  lower semi-continuous $v:\Omega_T\rightarrow \R$ is an {\sl $\e$-supperbarrier} to the $(CP_1)$ at the boundary point $(0,x_0)$, if $v$ is a supersolution to the $(CP_1)$ in $[0,T)\times X$, and $v^*(0,x_0)\leq \f_0(x_0)+\e$.\\
When $\e=0 $, $v$ is called a superbarrier.
\end{itemize}
\end{defi}
\begin{prop}\label{barrier}
Fix $\e>0$, there exist an $\e$-subbarrier and an $\e$-superbarrier to the Cauchy problem $ (CP_1)$ in $[0,T)\times \Amp(\alpha)$. 
\end{prop}
\begin{proof}
It  is straightforward  that the subsolution $u$ constructed in Lemma \ref{subsolution 1} is a subbarrier, so $\e$-subbarrier for all $\e>0$. 

\medskip
We now  find an $\e$-supperbarrier. Assume that $\theta\leq \omega$ for some K\"ahler form $\omega$ on $X$,  then $\f_0$ is also a $\omega$-psh function.  Suppose $h_j$ be a sequence of smooth function decrease to $\f_0$.  Denote  $\f_j=P(h_j):=\sup_X \{ \phi\in PSH(X,\omega)| \, \phi\leq h_j \}$ the envelope of $h_j$ then  $\f_j\searrow \f$ and $$(\omega+dd^c\f_j)^n=\mathbbm{1}_{\{\f_j=h_j\}} (\omega+dd^c h_j)^n=\mathbbm{1}_{\{\f_j=h_j\}} f_j\mu,$$
where $f_j\geq 0$ is bounded. Thus for any $\e$ sufficiently small and any $x_0\in \Amp(\alpha)$, there exists a function $\f_j$ such that 
$$ \f_0(x_0)\leq \f_j(x_0) \leq\f_0(x_0)+\e.$$
Define
$$v_j(t,x):=\f_j+Bt,$$ 
for a positive constant $B$ will be chosen hereafter,  then
\begin{align*}
(\theta+dd^c v_j)^n &=(\theta+dd^c\f_j)^n \\
&=\mathbbm{1}_{\{\f_j=h_j\}}(\theta+dd^ch_j)^n\\
&= \mathbbm{1}_{\{\f_j=h_j\}} f_j \mu
\end{align*}
where $f_j\geq 0$ is bounded. 

\medskip  
Since $h_j$ is smooth on $X$, there exists a constant $C_j>0$ such that 
$$\mathbbm{1}_{\{\f_j=h_j\}}f_je^{-\f_j}= \mathbbm{1}_{\{\f_j=h_j\}}f_je^{-h_j}\leq C_j.$$
We now choose $B>0$ such that $e^{B}\geq C_j$, hence 
\begin{eqnarray*}
(\theta+dd^cv_j)^n&=&\mathbbm{1}_{\{\f_j=h_j\}}f_j\mu\\
&=&\mathbbm{1}_{\{\f_j=h_j\}}f_je^{-\f_j}e^{\f_j}\mu\\
&\leq& C_je^{-B} e^{\dot v_j+v_j}\mu\\
&\leq& e^{\dot v_j+v_j}\mu.
\end{eqnarray*}
It follows from Lemma \ref{pluri and vis supersolution} that $v_j(t,x)$ is an $\e$-supperbarrier to $(CP_1)$ at  $(0,x_0)\in \{0\}\times \Amp(\alpha)$, so is 
$$v_{\e}:=\inf\{v_j(t,x),v+A_2t \},$$
where $v$ is the supersolution to $(CP_1)$ in Lemma \ref{supersolution 1}.
\end{proof}
\subsection{The Perron envelope}
Consider the upper envelope
$$\f:=\sup\{w, w \text{ is a subsolution of }(CP_1),u\leq w\leq v \},$$
where $u$ and $v$ are the viscosity  sub/super-solution from Lemma \ref{subsolution 1} and \ref{supersolution 1}.

\begin{thm}\label{thm:envelope}
The upper envelope $\f$ is the unique viscosity solution to $(CP_1)$ in $[0,T)\times \Omega$. 
\end{thm}
\begin{proof} Let $\f^*$ (resp. $\f_*$) be the upper (resp. lower) semi-continuous envelope for $\f$ in $[0,T)\times \Omega$, and  set $\f^*(t,x)=\f_*(t,x)=\f(t,x)$  in $[0,T)\times (X\setminus \Omega)$.
 Obeserve that $\f^*$ (resp. $\f_*$) is a subsolution (resp. supersolution) to the complex Monge-Amp\`ere flow 
$$(\theta+dd^c \phi)^n=e^{\partial_t \phi+\phi}\mu$$
on $(0,T)\times \Omega$.  We now show that they are also subsolution and supersolution respectively  to the Cauchy problem $(CP_1)$.

\medskip
We first have $\f\geq u$ on $\Omega_T$. Since $u$ is continuous in $[0,T)\times \Omega$, $\f_*(0,x_0)\geq \f_0(x_0)$ for any $x_0\in \Amp(\alpha)$. This shows that $\f_*$ is a supersolution to the Cauchy problem $(CP_1)$.
\medskip

We prove now that $\f^*(0,.)\leq \f_0$ in $\Omega$.  Fix $\e>0$, by Proposition \ref{barrier} there exists an $\e$-supperbarrier $v_\e$ to $(CP_1)$ at any point $(0,x_0)$ with $x_0\in \Omega$. It follows from the comparison principle (Theorem \ref{comparison_strict_positive})  
that
$$\f(t,x)\leq v_\e(t,x)\quad \text{in }[0,T)\times\Omega,$$  
hence
$$\f^*(0,x_0)\leq v_\e^*(0,x_0)\leq\f_0(x_0)+\e$$
for all $x_0\in \Omega$, hence $\f^*$ is a viscosity subsolution to the Cauchy problem $(CP_1)$.

\medskip
The comparison principle (Theorem \ref{comparison_strict_positive}) therefore implies that $\f^*=\f_*=\f$ in $[0,T)\times \Omega$.

\medskip
Finally, the uniqueness of viscosity solution in $[0,T)\times \Omega$ is deduced by the comparison principle (Theorem \ref{comparison_strict_positive}).
\end{proof}
\subsection{Long term behavior of the flow}
It follows from \cite{BEGZ} and \cite{EGZ15} that the Monge-Amp\`ere equation 
\begin{equation}\label{DMAE}
(\theta+dd^c\phi)^n=e^{\phi}\mu
\end{equation}
have a unique pluripotential solution which is a viscosity solution in $\Amp(\alpha)$. In this section,  we will prove that the solution of the Monge-Amp\`ere flow 
\begin{equation}\label{DMAF}
(\theta+dd^c\f_t)^n=e^{\dot{\f_t}+\f_t}\mu
\end{equation}
converges to the solution of (\ref{DMAE}) as $t\rightarrow +\infty$. 
\begin{thm}\label{convergence_1}
The solution $\f_t$ of the complex Monge-Amp\`ere flow  \ref{DMAF} starting at $\f_0$ converges,  exponentially fast in $\Amp(\alpha)$, as $t\rightarrow +\infty$, to the solution of the degenerate elliptic Monge-Amp\`ere equation (\ref{DMAE}).
\end{thm}

\begin{proof}
Let $\phi$ be the unique pluripotential solution to the equation (\ref{DMAE}).  Set  
 $$u(t,x):=e^{-t}\f_0 +(1-e^{-t})\phi +f(t),$$
 where $f(t)=O(te^{-t})$ is the unique solution of the ODE $f'(t)+f(t)=n\log(1-e^{-t})$ and $f(0)=0$. We now have $u(0,x)=
 \f_0(x)$ and
\begin{eqnarray*}
(\theta+dd^cu)^n&=&\big(e^{-t}(\theta+dd^c\f_0)+(1-e^{-t})(\theta+dd^c\phi)\big)^n\\
&\geq& (1-e^{-t})(\theta+dd^c\phi)^n\\
&=&e^{\partial_tu+u_t}\mu.
\end{eqnarray*} 
Lemma \ref{pluri and vis  subsolutions} implies that $u$ is a subsolution of $(CP_1)$ in $[0,T)\times \Amp(\alpha)$. It follows from the comparison principle (Theorem \ref{comparison_strict_positive}) we get $u\leq \f$ in $[0,T)\times \Amp(\alpha)$. Therefore
$$\f_t-\phi\geq e^{-t}(\f_0-\phi)+f(t) \quad \text{ in } [0,T]\times \Amp(\alpha).$$
 \medskip
 In addition, we also have $$v(t,x)=Ae^{-t}+\phi,$$
 where $A>0$ satisfies $|\f_0-\phi|\leq A$, is a supersolution of $(CP_1)$ in $[0,T)\times \Amp(\alpha)$. By the comparison principle (Theorem \ref{comparison_strict_positive}), we obtain $v(t,x)\geq \f_t$ in $[0,T)\times\Amp(\alpha)$, hence
 $$\f_t-\phi\leq Ae^{-t}$$
 in $[0,T)\times\Amp(\alpha)$.
 
 \medskip
 All together yields $|\f_t-\phi|=O(te^{-t})$ in $[0,T]\times \Amp(\alpha)$.
 Letting $t\rightarrow +\infty$ we obtain $\f_t\rightarrow \phi$ in $\Omega$.  
\end{proof}

\section{The K\"ahler-Ricci flow on manifolds of general tpye}\label{KFR_sect}
Let $X$ be a K\"ahler manifold  with the canonical bundle $K_X$ is big but not nef. Fix $\alpha_0$ a K\"ahler class on $X$ and $\omega_0$ is a K\"ahler form representing $\alpha_0$.  A (classical) solution of the {\sl normalized K\"ahler-Ricci flow} on $X$ starting at $\hat \omega\in \alpha_0$ is a family of K\"ahler forms  $(\omega_t)$ solving
\begin{equation} \label{NKRF}
\frac{\partial \omega_t}{\partial t}=-Ric(\omega_t)-\omega_t,\quad {\omega_t}_{|_{t=0}}=\hat \omega.
\end{equation}

Note that $\omega_t\in \alpha_t:= e^{-t}\alpha_0+(1-e^{-t})c_1(K_X)$. Moreover it follow from  \cite{TZ06} (see also \cite{Cao85,Tsu88} for special cases) that  
the normalized K\"ahler-Ricci flow with initial metric $\omega_0$ has a unique smooth solution on on $[0,T)$ with 
$$T:=\sup \{t>0|\;  e^{-t}\alpha_0+(1-e^{-t})c_1(K_X)>0 \}.$$
Assuming that $\alpha_T$ is big, Tosatti and Collins  \cite{CT} proved  that as $t\rightarrow T^-$ the metric $\omega(t)$ develop singularities precisely on the Zariski closed set $ X\setminus \Amp(K_X)$.

\medskip
Now at $T$, $\alpha_{T}$ is big and nef. However, for $t>T$, $\alpha_t$ is still big but no longer nef, thus we can not continue the flow in the classical sense.  In  \cite{FIK} the authors asked  whether one can  define and construct weak solutions of K\"ahler-Ricci flow after  the maximal existence time for smooth solutions.
 In \cite{BT12}, Boucksom and Tsuji have tried to run the   normalized K\"ahler-Ricci on projective varieties  in a weak sense  beyond the maximal time using the  discretization of the K\"ahler-Ricci flow  and algebraic geometry tools. They have proposed a conjecture  (cf. \cite[Conjecture 1, page 208]{BT12}) with respect to this direction. 

\medskip
In this section  we  answer the question of  a Feldman-Ilmanen-Knopf and give an analytic approach to the conjecture of Boucksom and Tsuji using  the viscosity theory established in Section \ref{sec:comparison_principle}.   Moreover we  show that the weak flow exists for all time and  converges to the singular K\"ahler-Einstein metric contructed in \cite{BEGZ}.  

\subsection{Existence and uniqueness of extended flow}
Now let $\theta$ be a smooth closed $(1,1)$-form representing $c_1(K_X)$ and set
$$\theta_t:=e^{-t}\omega_0+(1-e^{-t})\theta.$$
Let $dV$ be a smooth volume form on  $X$, then $-{\rm Ric}(dV)\in c_1(K_X)$. Therefore there exists $f\in C^\infty (X)$ such that $\theta= {\rm Ric}(\mu)$ with $\mu=e^f dV $.  Then  the normalized K\"ahler-Ricci flow (\ref{NKRF}) can be written as the complex Monge-Amp\`ere flow
$$(CP_2)\quad  \begin{cases}
(\theta_t+dd^c\f_t)^n=e^{\dot{\f_t}+\f_t} \mu\\
\f(0,x)=\f_0
\end{cases},
$$
where $\hat \omega=\omega_0+dd^c\f_0$. 
\begin{lem}\label{subsolution 2}
For any $T>0$, there exists a subsolution and a supersolution to the Cauchy problem $(CP_2) $ in $[0,T)\times \Amp(K_X)$.

\end{lem}
\begin{proof}
It follows from \cite{Bou04} that there exists a $\theta$-psh function $\rho$ with analytic singularities  satisfying
$$
\theta+dd^c\rho\geq \epsilon \omega_0  \quad \text{and }\sup_X\rho=0.
$$
We consider $$u(t,x):=e^{-t}\f_0+(1-e^{-t})\rho+f(t) -At,$$
where $f(t)$ is the unique solution of $f'(t)+f(t)=n\log (1-e^{-t})$ and $f(0)=0$. Then $u(0,x)=\f_0(x)$ and $u$ is a $\theta_t$-psh function since 
\begin{eqnarray*}
\theta_t+dd^cu&=&e^{-t}\omega_0+(1-e^{-t})\theta +dd^cu\\
&=& e^{-t}\left(\omega_0+dd^c\f_0\right)+(1-e^{-t})(\theta+dd^c\rho)\geq 0.
\end{eqnarray*}
Therefore
Then  $u$ is continuous on $[0,T)\times  \Omega$ and    $u(0,x)=
 \f_0(x)$ and
\begin{eqnarray*}
(\theta_t+dd^c  u)^n&=&\left( e^{-t}(\omega_0+dd^c\f_0)+(1-e^{-t})(\theta+dd^c\rho)\right)^n\\
&\geq& (1-e^{-t})^n(\theta+dd^c\rho)^n\\
&\geq &e^{n\log (1-e^{-t})} \epsilon^n \omega^n  \\
&=& e^{\partial_t u+u_t -\rho +A  +\log(\epsilon C) }\mu,\\
&\geq & e^{\partial_t u+u_t +A  +\log(\epsilon C) }\mu,
\end{eqnarray*} 
where $C^n\mu\leq \omega^n$. By choosing $A=-\log(\epsilon C)$, we infer that 
$ (\theta_t+dd^cu)^n\geq e^{\dot{u}_t +u_t}\mu$
in the pluripotential sense.  Lemma \ref{pluri and vis  subsolutions} thus  implies that $u$ is a subsolution to $(CP_2)$ in $[0,T)\times\Amp(K_X)$.
\medskip

For supersolution, we suppose that $\theta\leq A\omega_0$ for some $A>0$. Denote $\psi$ is the unique solution of the complex Monge-Amp\`ere equation 
$$(\omega_0+dd^c\psi)^n=c\mu,\quad \min_X \psi =0,  \text{ with } c=\frac{\int_X \omega_0^n}{\int_X \Omega}.$$
Then we set $v= a(t)\psi+Bt$, where  $a(t)= e^{-t}+A(1-e^{-t})$ and $g$ will be chosen hereafter. We have
\begin{eqnarray*}
(\theta_t+dd^cv)^n&\leq&(a(t)\omega_0+dd^cv )^n\\
&=&e^{\log c +n\log a(t)}\Omega
\end{eqnarray*}
  By choosing $$B=\log c+ n\max_{[0,T]} \log a(t),$$ we have   $\log c +n\log a(t) \leq \dot{v}+v$,    Lemma \ref{pluri and vis supersolution} thus  implies that  $v$ is a supersolution to $(CP_2)$. 
\end{proof}
\begin{lem}
Fix $\e>0$. There exist an $\e$-subbarrier and an $\e$-supperbarrier to the Cauchy problem $(CP_2)$.
\end{lem}
\begin{proof}
Observe that for any $\e>0$, the subsolution $u(t,x)$  in Lemma \ref{subsolution 2} is also a $\e$-subsolution since $u(0,x)=\f_0$.  

\medskip
For the $\e$-supperbarrier at $(t_0,x_0)\in \{t_0\}\times \Amp(K_X)$, we use the same argument in Proposition \ref{barrier}.  Approximate $\f_0$ by a decreasing sequence $(h_j)$ of smooth functions. Denote  $\f_j=P(h_j)$ the envelope of $h_j$ then  $\f_j\searrow \f_0$ and $$(\omega_0+dd^c\f_j)^n=\mathbbm{1}_{\{\f_j=h_j\}} (\omega_0+dd^c h_j)^n=\mathbbm{1}_{\{\f_j=h_j\}} f_j\mu,$$
where $f_j\geq 0$ is bounded.  Thus there exists a function $\f_j$ such that 
$$ \f_0 \leq \f_j \leq\f_0 +\e.$$

Define $v(t,x)=a(t)\f_j+Bt$ with $a(t)$ as in Lemma \ref{subsolution 2}, then
\begin{eqnarray*}
(\theta_t+dd^cv)^n&\leq&(a(t)\omega_0+dd^cv )^n\\
&=&\mathbbm{1}_{\{\f_j=h_j\}}f_je^{n\log a(t) }\mu\\
&\leq& \mathbbm{1}_{\{\f_j=h_j\}}f_j e^{(A-1)\f_j+n\log a(t)-B} e^{\partial_t v+v}\mu.
\end{eqnarray*}
Since $h_j$ is smooth on $X$ and $f_j$ is bounded on $X$, there is a constant $C_j$ such that 
$$\mathbbm{1}_{\{\f_j=h_j\}}f_je^{(A-1)h_j}\leq e^{C_j}.$$
By choosing $B=\max_{[0,T]}n\log a(t)+C_j$, we get 
$$(\theta_t+dd^cv)^n \leq e^{\partial_tv_t+v_t}\mu.$$
Since $\f_j$ is continuous on $X$, so is $\dot{v}+v$, Lemma \ref{pluri and vis supersolution} infers that $ v$ is a $\e$-superbarrier to the Cauchy problem $(CP_2)$ in $[0,T)\times \Amp(K_X)$. 
\end{proof}

We now obtain the solution of $(CP_2)$ using the Perron envelope as in Theorem \ref{thm:envelope}.
\begin{thm}\label{existence of KRF}
For any $T>0$, there exists a unique viscosity solution to the Cauchy problem $(CP_2)$ on $[0,T)\times X$. As consequence, the normalized K\"ahler-Ricci flow exists for all time  in the viscosity sense.
\end{thm}

\begin{proof}
We consider the
 upper envelope
$$\f:=\sup\{w, w \text{ is a subsolution of }(CP_2),u\leq w\leq v \},$$
where $u$ (reps. $v$) is the subsolution (reps. supersolution)  of $(CP_2)$ contructed above. 
Let $\f^*$ (resp. $\f_*$) be the upper (resp. lower) semi-continuous envelope for $\f$ in $[0,T)\times\Amp(K_X)$, and  set $\f^*(t,x)=\f_*(t,x)=\f(t,x)$,   $\forall(t,x)\in [0,T)\times (X\setminus \Amp(K_X))$. 

\medskip
Then we have $\f^*(x)$ is a viscosity subsolution to the equation
\begin{equation}\label{eq:MAF_KFR}
(\theta_t+dd^c\f_t)^n=e^{\partial_t\f_t+\f_t}\mu
\end{equation}
on $[0,T)\times X$. In addition,  it follows from the bump construction (cf. \cite{CIL92,EGZ11}) that $\f_*$ satisfies the viscosity inequality in Definition \ref{def:super}. 

\medskip
To  see that $\f_*$ is a supersolution of \eqref{eq:MAF_KFR}, we need to prove further that $ \f(t,x)\geq V_{\theta_t}(x)-C(t), \forall t\in [0,T)$, for some time-dependent constant  $C(t)$ (cf. Definition \ref{def:super}). It is sufficient to find a subsolution $\psi_t$ of $(CP_2)$ such that  $\psi_t\geq V_{\theta_t}(x)-C(t), \forall t\in [0,T)$. This  is straightforward in Theorem \ref{thm:envelope} when $\theta_t$ is independent of $t$, but not trivial when $\theta_t$ depends on $t$. We would like to thank Hoang-Chinh Lu for pointing out this missing argument in our last version. We now give such supersolution in our case when $\theta_t=e^{-t}\omega_0+ (1-e^{-t})\theta $. 

\medskip
We first remark that $$\chi_t:=\frac{\theta_t}{1-e^{-t}}= \theta+ \frac{\omega_0}{e^t-1}$$  is decreasing in $t$.
%

\medskip
 Let $\phi_t(x)$ be the unique elliptic viscosity solution of 
\begin{equation}
(\chi_t+dd^c\phi_t)^n =e^{\phi_t}\mu,
\end{equation}
for each $t\in [0,T)$.
Since $t\mapsto \chi_t$ is decreasing, for $s\leq t$ we have 
$$(\chi_s+ dd^c \phi_t)^n\geq (\chi_t+ dd^c\phi_t)^n =e^{\phi_t}\mu .$$
Therefore $\phi_t$ is a viscosity subsolution to $(\chi_s+dd^c \phi_s)^n =e^{\phi_s}\mu$. It follows from the comparison principle (cf. Theorem \ref{CP})  that $\phi_t\leq \phi_s$ in $\Amp(\{\chi_s\})\supset \Amp(K_X)$.  Hence  $t\mapsto \phi_t$ is decreasing on $[0,T)\times\Amp(K_X)$. 

\medskip
Set $\psi_t:=(1-e^{-t})\phi_t + f(t)+A$, where $A= \min\{\inf_X \f_0, 0\}$, and  $f$ satisfies the ODE: $f'(t)+f(t)= n\log (1-e^{-t})$ and  $f(0)=0$.  Then we have
\begin{eqnarray*}
\partial_t \psi_t +\psi_t &=& (1-e^{-t})\partial_t\phi_t  + e^{-t}\phi_t +(1-e^{-t})\phi_t+ n\log (1-e^{-t}) +A\\
&=&(1-e^{-t})\partial_t\phi_t  +\phi_t+ n\log (1-e^{-t}) +A
\end{eqnarray*}
Since  $t\mapsto \phi_t$ is decreasing on  $[0,T)\times\Amp(K_X)$ and $A\leq 0$, we infer that
\begin{eqnarray*}
(\theta_t+dd^c\psi_t)^n &=& (1-e^{-t})^n (\chi_t+dd^c\phi_t)^n  \\
&=& e^{\phi_t+n\log(1-e^{-t})}\mu\\
&\geq & e^{(1-e^{-t})\partial_t\phi_t + \phi_t +n\log(1-e^{-t})+A}\mu\\
&= &e^{\partial_t\psi_t + \psi_t }\mu
\end{eqnarray*}
in viscosity sense.
This follows that $\psi_t$ is a viscosity subsolution to $(CP_2)$. Since $\phi_t\geq V_{\chi_t}-C(t)$ for $t\in [0,T)$, we have  $\psi_t\geq V_{\theta_t}-C'(t)$ for  $t\in [0,T)$ as required.

\medskip
Finally, as in the proof of Theorem \ref{thm:envelope}, we use the $\e$-sub/supper-barriers constructed above to prove that $\f^*$ (resp. $\f_*$) is the subsolution (resp. suppersolution) to $(CP_2)$.  Then the comparison principle (Corollary \ref{cor:comparison})  implies that $\f^*=\f_*=\f$ in $[0,T)\times \Amp (K_X)$, hence $\f$ is a viscosity solution to $(CP_2)$. The uniqueness again follows from the comparison principle (Corollary \ref{cor:comparison}).
\end{proof}
\subsection{Convergence of the weak normalized K\"ahler-Ricci flow}
We now study the long-time behavior of the normalized K\"ahler-Ricci flow on compact K\"ahler manifolds of general type. Precisely we prove that the normalized K\"ahler-Ricci flow continuously deforms any initial K\"ahler form $\omega_0$ towards the unique singular K\"ahler-Einstein metric $\omega_{KE}=\theta+dd^c\f_{KE}$ in the canonical class $K_X$, with 
\begin{equation}
(\theta+dd^c\f_{KE})^n=e^{\f_{KE}}\mu
\end{equation}
(we refer to \cite{EGZ09} and \cite{BEGZ} for the construction of $\omega_{KE}$). 
\begin{thm}\label{convergence of KRF}
The viscosity  solution of  the Monge-Amp\`ere flow
$$\begin{cases}
(\theta_t+dd^c\f_t)^n=e^{\dot{\f_t}+\f_t}\mu\\
\f(0,x)=\f_0
\end{cases}
$$
 converges,  as $t\rightarrow +\infty$, locally uniformly on $\Amp(K_X)$ to the unique solution $\f_{KE}$ of the Monge-Amp\`ere equation
$$(\theta+dd^c\f_{KE})^n=e^{\f_{KE}}\mu.$$
\end{thm} 
\begin{proof}
Set
$$u(t,x)=e^{-t}\f_0+(1-e^{-t})\f_{KE}+f(t)$$
where  $f(t)=O(te^{-t})$  is the unique solution of the ODE $f'(t)+f(t)=n\log(1-e^{-t})$ and $f(0)=0$ (cf. Lemma \ref{lem:ODE}). 
Then $u$ is a $\theta_t$-psh function and
\begin{eqnarray*}
(\theta_t+dd^c  u)^n&=&\left( e^{-t}(\omega_0+dd^c\f_0)+(1-e^{-t})(\theta+dd^c\f_{KE})\right)^n\\
&\geq& (1-e^{-t})^n(\theta+dd^c\f_{KE})^n\\
&\geq &e^{n\log (1-e^{-t}) +\f_{KE}} \mu \\
&= &e^{\dot{u}+u} \mu.
\end{eqnarray*}
Lemma \ref{pluri and vis  subsolutions} implies that  $u$ is a viscosity subsolution.  
It follows from   Theorem \ref{comparison_strict_positive}  that 
\begin{equation}
\f_t\geq u\quad \text{on}\,\, [0,+\infty)\times \Amp(K_X),
\end{equation}
 hence
 \begin{equation}\label{lower_bound}
 \f_t-\f_{KE}\geq e^{-t}(\f_0-\f_{KE})+f(t),
 \end{equation}
on $[0,\infty)\times \Amp(K_X)$ with $f(t)=O(te^{-t})$.

\medskip
For the upper bound of $\f_t-\f_{KE}$ we need to use the following lemma 

\begin{lem}\label{lem:converge}
There exists a unique viscosity solution $\phi_t$ for the following flow
\begin{equation}\label{CMAF_2}
\begin{cases}
\big(a(t)\theta+dd^c\phi_t\big)^n=e^{\partial_t\phi_t+\phi_t}\mu\\
\phi(0,x)=\phi_0
\end{cases}
\end{equation}
for any $\phi_0\in PSH(X,(1+A)\theta)\cap C^0(\Amp(K_X))$ with minimal singularities,
where $a(t)=1+Ae^{-t}$ with $A\geq 0$. Moreover, the flow converges to $\f_{KE}$, locally uniformly on $\Amp(K_X)$  as $t\rightarrow +\infty$. 
\end{lem}
\begin{proof}
Observe that
\begin{eqnarray}
\big(a(t)\theta+dd^c\phi_t\big)^n=a^n(t)\big(\theta+dd^ca(t)^{-1}\phi_t\big)^n .
\end{eqnarray}
By setting $\tilde{\phi}_t=a(t)^{-1}\phi_t -b(t)$, with $b(t)=a(t)^{-1}v(t)$ where $v=O(te^{-t})$ is the unique solution of  the ODE
$$v'+v=n\log a(t)\quad \text{and }\, v(0)=0.$$
We now can rewrite (\ref{CMAF_2}) to the flow  
\begin{equation}\label{CMAF_3}
\begin{cases}
\big(\theta+dd^c\tilde\phi_t\big)^n=e^{a(t)\partial_t\tilde\phi_t+\tilde\phi_t}\mu\\
\tilde \phi(0,x)=\phi_0
\end{cases}.
\end{equation}
Finally, by changing of the time variable $\psi(t,x)=\tilde{\phi}(h(t),x)$ where $h(t)$ is the unique solution of  the ODE
$$ h'(t)=a(h(t)) \quad\text{and }\, h(0)=0.$$ 
Then the equation (\ref{CMAF_2}) can be rewritten 
as
\begin{equation}\label{CMAF_4}
\begin{cases}
(\theta+dd^c \psi_t)^n=e^{\partial_t\psi+\psi_t}\mu\\
\phi(0,x)=\phi_0,
\end{cases}
\end{equation}
which is the flow we studied in Section \ref{Cauchy_pro_sect}.  Since  $b(t)\rightarrow 0$ and $h(t)\rightarrow +\infty$ as $t\rightarrow \infty$ the convergence is followed from  Theorem \ref{convergence_1}.
 \end{proof}
Since $K_X$ is big there exists a $\theta$-psh function $\rho$ with analytic singularities  satisfying
\begin{equation}\label{big_character}
\theta+dd^c\rho\geq C^{-1} \omega_0  \quad \text{and }\sup_X\rho=0,
\end{equation}
for some $C>0$ (cf. \cite{Bou04}). We can assume  further that $C\geq 1$. 

\medskip
Now set $u(t,x)=\f_t-Be^{-t}+Ce^{-t}\rho,$ and $a(t)=1+(C-1)e^{-t}$.
Using (\ref{big_character}) we have
\begin{eqnarray*}
a(t)\theta+dd^cu&=&Ce^{-t}\theta+(1-e^{-t})\theta+dd^c u\\
&\geq& e^{-t}(\omega_0-Cdd^c \rho)+(1-e^{-t})\theta+dd^c u\\
&=&\theta_t+dd^c\f_t,
\end{eqnarray*}
hence 
$$ (a(t)\theta+dd^cu)^n\geq e^{\partial_t u+u}\mu$$
in the viscosity sense.  Fix  $\phi_0$ a  $C\theta$-psh function with minimal singularities, then $ \phi_0-C\rho$ is bounded from below.  Therefore we can choose $B>0$ such that  $\phi_0 -C\rho \geq  \f_0   - B $. 
This implies that $u$ is a subsolution of the Cauchy problem (\ref{CMAF_2}). Since the flow (\ref{CMAF_2}) can be written as the flow (\ref{CMAF_4}) after changing of time variable,  the comparison principle also holds  for the flow (\ref{CMAF_2}). Therefore  we get 
$$u\leq \phi_t $$ on $ [0,\infty)\times \Amp(K_X)$.
  Combining with  (\ref{lower_bound})  and Lemma \ref{lem:converge}, we imply that  $\f_t$ converges to $\f_{KE}$ on $\Amp(K_X)$. 
\end{proof}

\end{document}